\RequirePackage{fix-cm}
\documentclass[smallextended]{svjour3}       
\smartqed  
\usepackage{graphicx}
\usepackage{mathtools}
\usepackage{amsmath, amssymb}
\begin{document}

\title{Almost order-weakly compact operators on Banach lattices
}


\author{Mina Matin         \and
        Kazem Haghnejad Azar \and
        Ali Ebadi
}


\institute{Mina Matin \at
	{Department  of  Mathematics  and  Applications, Faculty of Science, University of Mohaghegh Ardabili, Ardabil, Iran} \\
	\email{minamatin1368@yahoo.com}
\and
Kazem haghnejad Azar \at
	{Department  of  Mathematics  and  Applications, Faculty of Science, University of Mohaghegh Ardabili, Ardabil, Iran} \\
	\email{haghnejad@uma.ac.ir}
\and
Ali Ebadi \at
	{Department  of  Mathematics  and  Applications, Faculty of Science, University of Mohaghegh Ardabili, Ardabil, Iran} \\
	\email{ali1361ebadi@gmail.com.ac.ir}
}

\date{Received: date / Accepted: date}

\maketitle

\begin{abstract}
A continuous operator $T$ between two Banach lattices $E$ and
	$F$ is called almost order-weakly compact,
	whenever for each almost order bounded subset $A$ of $E$,  $T(A)$ is a relatively weakly compact subset of $F$. In Theorem \ref{ree}, we show that the positive operator $T$ from  $E$ into    Dedekind complete $F$  is almost order-weakly compact if and only if  $T(x_n) \xrightarrow{\|.\|}0$ in $F$ for each disjoint almost order bounded sequence $\{x_n\}$ in $E$.
	In this manuscript, we study some
	properties of this class of operators and its relationships with others known operators.
\end{abstract}

\section{Introduction}
Since order weakly compact operators play important role in positive operators, our aim in this manuscript is to introduce and study a new class of operators as almost order-weakly compact operators and we establish some of its relationships with others known operators. Under some conditions, we show that the adjoint of any almost order-weakly compact operator is so. Every compact and weakly compact operators are almost order-weakly compact operator, but the converse in general not holds.

To state our results, we need to fix some notations and recall some definitions.
Let $E$ be a Banach lattice.  A subset $A$ is said to be almost order bounded if for any $\epsilon$ there exists $u\in E^+$ such that $A\subseteq [-u,u]+\epsilon B_E$ ($B_E$ is the closed unit ball of $E$). One should observe the following useful fact, which can be easily verified using Riesz decomposition Theorem, that $A\subseteq [-u,u] + \epsilon B_E$ iff $\sup_{x\in A}\|(|x|-u)^+\| = \sup_{x\in A} \| |x| -|x|\wedge u\| \leq \epsilon$. By Theorems 4.9 and 3.44 of \cite{1}, each almost order bounded subset in order continuous Banach lattice is relatively weakly compact.  $A\subseteq L_1(\mu)$ is relatively weakly compact iff it is almost order bounded (see\cite{9c}).
Recall that a vector $e > 0$ in vector lattice $E$ is an order unit or a strong unit (resp, weak unit) when the ideal $ I_{e}$ (resp, band $B_e$) is 
equal to $ E $; equivalently, for every $ x \geq 0 $ there exists
$ n \in \mathbb{N}$ such that  $ x \leq ne$ (resp,  $x\wedge ne \uparrow x$ for every $x\in E^+$).
Suppose that Banach lattice $E$ is an order continuous norm with a weak unit $e$. It is known that $E$ can be represented as a norm and order dense ideal in $L_1(\mu)$ for some finite measure $\mu$ (see \cite{5}). 
	A continuous operator $T$ from a Banach lattice $E$ to a Banach space $X$ is said to be
\begin{itemize}
	\item 
	\textbf{order weakly compact} whenever $T[0,x]$ is a relatively weakly compact subset of $X$ for each $x\in E^+$.
	\item 
	\textbf{$M$-weakly compact} if $T(x_n)\xrightarrow{\|.\|}0$ holds for every norm bounded disjoint sequence $\{x_n\}$ of $E$.
	\item 
	\textbf{$b$-weakly compact} whenever $T$ carries each $b$-order bounded subset of $E$ into a relatively weakly compact subset of $X$.
\end{itemize}
A continuous operator $T$ from a Banach space $X$ to a Banach lattice $E$ is said to be
\begin{itemize}
	\item 
	\textbf{$L$-weakly compact}  whenever $y_n\xrightarrow{\|.\|}0$ for every disjoint sequence $\{y_n\}$ in the solid hull of $T(B_X)$.
	\item \textbf{semicompact} whenever
	for each $\epsilon \geq 0$ there exists some $u\in E^+$ satisfying $\|(|Tx|-u)^+\|\leq \epsilon$ for all $x\in B_X$. 
\end{itemize}
An operator
$T : E \to F$ is regular if $T = T_1 - T_2$ where $T_1, T_2 : E \to F$ are positive operators. We denote by $L(E, F)$ ($L^r(E,F)$) the space of all operators (regular operators) from $E$ into $F$.

An operator $T:E\to F$ between two vector lattices is said to be lattice homomorphism (resp, preserve disjointness) whenever $T(x\vee y) = T(x)\vee T(y)$ (resp, $x \perp y$ in $E$ implies $T(x)\perp T(y)$ in $F$).

Recall that $L_b(E,F)$  is the vector space of all order bounded operators from $E$ to $F$.

A Banach space $X$ is said to be Grothendieck space whenever $weak^*$ and $weak$ convergence of sequences in $X^\prime$ (norm dual of $X$) coincide.

A Banach lattice $E$ is said to be $AM$-space (resp, $AL$-space), if for $x,y \in E$ with $x\wedge y =0$, we have $\|x\vee y\| = \max\{\|x\|,\|y\|\}$ (resp, $\|x+y\| = \|x\| + \|y\|$). A Banach lattice $E$ is said to be $KB$-space whenever every increasing norm bounded sequence of $E^+$ is norm convergent.

Let $E$ be a vector lattice and $x\in E$.
A net $\{x_{\alpha}\}\subseteq E$ is said to be unbounded order convergent to $x$ if
$ \vert  x_{\alpha} - x \vert  \wedge u \xrightarrow{o} 0$ for all
$ u \in E^{+}$. We denote this convergence by
$ x_{\alpha} \xrightarrow{uo}x $ and write that
$ \{x_{\alpha}\}_{\alpha} $ is $uo$-convergent to $x$.
\section{almost order bounded operators}
	Let $T:E\to F$ be a continuous operator between two Banach lattices. $T$ is said to be almost order-weakly compact operator whenever $T$ maps the almost order bounded subset $A$ of $E$ into an almost order bounded subset of $F$.\\
	The vector space of all almost order bounded operators from $E$ to $F$ will be denoned $L_{aob}(E,F)$.

It is obvious that if $T:E\to F$ is a semicompact operator, then it is almost order bounded. If $E$ has an order unit and $T:E\to F$ is order bounded, then it is an almost order bounded  operator and if $F$ has an order unit and $T$ is an almost order bounded opertor, then it is order bounded.

Here is an example show that the class of order bounded operators differ from the class of almost order bounded.
\begin{example}
The operator $T:L_1[0,1]\to c_0{\prime\prime}$ defined by
	\[
	T(f) = (\int_0^1f(x)sinxdx, \int_0^1 f(x) sin2x dx,\ldots)
	\]
	is not order bounded(see page 67 of \cite{1}). Let $A\subseteq L_1[0,1]$ be an almost order bounded. Hence $A$ is a relatively weakly compact subset of $E$. Because $T$ is continuous, so $T(A)$ also is relatively weakly compact. Since $c_0^{\prime\prime}$ is an $AL$-space, therefore by Theorem 4.27 of \cite{1}, it is lattice isometric by $L_1(\mu)$. Hence $T(A)$ is an almost order bounded subset of $c_0^{\prime\prime}$. So $T$ is an almost order bounded operator.
\end{example}
Here is an example that the operator $T$ is almost order bounded while whose modulus does not exist.
\begin{example}
Consider the continuous function $g:[0,1]\to [0,1]$ defined by $g(x) = x$ if $0 \leq x \leq \dfrac{1}{2}$ and $g(x) = \frac{1}{2}$ if $\frac{1}{2}< x \leq 1$. Now define the operator $T:C[0,1]\to C[0,1]$ by $Tf(x = f(g(x))) - f(\dfrac{1}{2})$. $T$ is a regular operator and therefore it is an order bounded operator. Because $C[0,1]$ is an $AM$-space with unit, so $T$ is an almost order bounded. Note that the modulus of $T$ does not exist(see Exercise 9 of page 22 of \cite{1}.).
\end{example}
We are looking for situations where if $T$ is an almost order bounded, then $\mid T \mid $ exist and it is an almost order bounded operator.
\begin{proposition}
Let  $T:E\to F$ be an almost order bounded operator between two Banach lattices that $F$ is Dedekind complete and $E, F$ havean order unit, then the modulus of $T$ exists and it is almost order bounded.	
\end{proposition}
\begin{proof}
	Let $T$ be almost order bounded. Since $F$ has an order unit, therefore $T$ is an order bounded operator. Because $F$ is Dedekind complete, so by Theorem 1.18  of \cite{1}, $\mid T\mid$ exist and it is an order bounded. Since $E$ has an order unit, hence $\mid T \mid$ is an almost order bounded.
	\end{proof}
\begin{proposition}\label{ii}
	If $T:E\to F$ is an onto lattice homomorphism, then $T$ is almost order bounded.	
\end{proposition}
\begin{proof}
	Let $T:E\to F$ be an almost order bounded and $A\subseteq E$ be an almost order bounded set. It means that for each $\varepsilon > 0$ there exists $u\in E^+$ that $\sup_{x\in A}\|(\mid x \mid - u)^+\|\leq \varepsilon$. Since $T$ is a positive operator therefore it is a continuous operator. Hence we have  for each $\varepsilon > 0$ there exists $u\in E^+$ that $\sup_{x\in A}\| T(\mid x \mid - u)^+\|\leq \varepsilon$. Since $T$ is a lattice homomorphism, therefore $\sup_{x\in A}\| (\mid T x \mid - Tu)^+\|\leq \varepsilon$ $\sup_{x\in A}\| T(\mid x \mid - u)^+\|\leq \varepsilon$. Because $T$ is onto, so the proof is complete.
\end{proof}
\begin{remark}
	If $T:E\to F$ is onto lattice homomorphism and $F$ is Archimedean, then $\mid T \mid$ exists and it is an almost order bounded.
	\end{remark}
\begin{proof}
	Since $T$ is lattice homomorphism, therefore it is an order bounded and disjointness preserving. Hence by Theorem 2.40 of \cite{1}, $\mid T \mid$ exists. It is obvious thst $\mid T \mid$ is a lattice homomorphism. By Proposition\ref{ii}, $\mid T \mid$ is an almost order bounded.
	\end{proof}
\section{ almost order-weakly compact operators}
	Let $T:E\to F$ be a continuous operator between two Banach lattices. $T$ is said to be almost order-weakly compact operator (for short, $ao$-$wc$ operator) whenever $T$ maps the almost order bounded subset $A$ of $E$ into a  relatively weakly compact subset of $F$.

By Theorem 3.40 of \cite{1}, $T$ is $ao$-$wc$ operator iff for every almost order bounded sequence $\{x_n\}$ of $E$ the sequence $\{T(x_n)\}$ has a weak convergent subsequence in $F$.

The collection of all $ao$-$wc$ operators between two Banach lattices $E$ and $F$ will be show by $K_{ao-wc}(E,F)$.

It is obvious that each compact and weakly compact operator is $ao$-$wc$ and each $ao$-$wc$ operator is an order weakly compact operator.

By Theorem 5.23 and 5.27 of   \cite{1}, we have the following result.
\begin{theorem}\label{gf}
	\begin{enumerate}
		\item Each continuous operator $T$ from a Grothendieck Banach lattice $E$ into a Banach lattice $F$ is an $ao$-$wc$ operator.
		\item Let $T$ be a positive operator from a Banach lattice $E$  into a Banach lattice $F$ and  $E^\prime$ has order continuous norm. If $F$ is a $KB$-space,
		then $T$ is $ao$-$wc$.
	\end{enumerate}
\end{theorem}

In the following we have some examples of $ao$-$wc$ operators.
\begin{example}\label{elin}
	\begin{enumerate}
		\item 
		Since  $C[0,1]$ is a Grothendieck space, therefore by Theorem \ref{gf}(1), the continuous operator $T:C[0,1] \rightarrow c_0$, given by 
		$$ T(f) = (\int_0 ^1 f(x)\sin x dx, \int_0 ^1 f(x)\sin2x dx,\cdots),$$
		is an $ao$-$wc$ operator.
		\item 
		Since $c^\prime$ has order continuous norm and $\mathbb{R}$ is a $KB$-space, therefore by Theorem \ref{gf}(2), the functional $f:c\to \mathbb{R}$ defined by
		$$f(x_1,x_2,...)=\lim_{n\rightarrow \infty} x_n$$
		is an $ao$-$wc$ operator.
	\end{enumerate}
\end{example}
\begin{proposition}\label{lable}
	Let $E$, $F$ and $G$ be three Banach lattices,  $T:E\to F$ and $S:F\to G$ be two $ao$-$wc$ operators. By one of the following conditions, $S\circ T$ is an $ao$-$wc$ operator.
	\begin{enumerate}
		\item $F$ is an $AL$-space.
		\item $F$ has order continuous norm with a weak unit.
	\end{enumerate}	
\end{proposition}
\begin{proof}
	Let $A\subseteq E$ be almost order bounded. By assumption,  $T(A)$ is relatively weakly compact subset of $F$. If $F$ is an $AL$-space, then by Theorem 4.27 of \cite{1}, $F$ is lattice isometric to some concrete $L_1(\mu)$ and if $F$ has order continuous norm with a weak unit,  then $F$ is norm and order dense ideal in $L_1(\mu)$. Therefore $T(A)$ is an almost order bounded subset of $F$. So by assumption, $S(T(A))$ is relatively weakly compact subset of $G$. Hence $S\circ T$ is an $ao$-$wc$ operator.
\end{proof}
As following example the adjoint of $ao$-$wc$ operator in general is not $ao$-$wc$ operator.
\begin{example}
	Let $A\subseteq\ell^1$ be an almost order bounded set. Since $\ell^1$ has order continuous norm, therefore $A$ is relatively weakly compact. Thus the identity operator $I:\ell^1 \to \ell^1$ is an $ao$-$wc$ operator. Since the identity operator $I:\ell^\infty \to \ell^\infty$ is not order weakly compact, therefore  is not $ao$-$wc$.
\end{example}
In the following theorem, under some conditions, we show that the adjoint of $ao$-$wc$ operator is so.
\begin{theorem}\label{eli}
	Let $T:E\to F$ be an $ao$-$wc$ operator between two Banach lattices. If any of the following conditions are met, then $T^\prime$ is $ao$-$wc$.
\end{theorem}
\begin{enumerate}
	\item $E$ has an order  unit.
	\item $E^\prime$ is a $KB$-space and $F^\prime$ has an order unit.
\end{enumerate}
\begin{proof}
	\begin{enumerate}
		\item Let $E$ has an order unit and $T:E\to F$ be $ao$-$wc$. If $A\subseteq E$ is norm bounded, then $A$ is an order bounded and therefore almost order bounded. Hence by assumption, $T(A)$ is a relatively weakly compact subset of $F$. It means that $T$ is a weakly compact operator. Therefore by Theorem 5.5 of \cite{2}, $T^\prime$ is weakly compact and hence it is an $ao$-$wc$ operator.
		\item Let  $T:E\to F$ be an $ao$-$wc$ operator. Therefore $T$ is an order weakly compact operator. Since $E^\prime$ is a $KB$-space, by Theorem 3.3  of \cite{2c}, $T^\prime$ also is an order weakly compact operator. Since $F^\prime$ has an order unit, it is clear that $T^\prime$ is $ao$-$wc$.
	\end{enumerate}	
\end{proof}
We know that each compact and weakly compact operator is an $ao$-$wc$ operator, but by following example the converse in general not holds.
\begin{example}
 The identity operator $I:\ell^1 \to \ell^1$ is an $ao$-$wc$ operator but  is not compact or weakly compact operator.
\end{example}
\begin{corollary}\label{hg}
Under conditions of Theorem \ref{eli}, an operator $T:E\to F$ is weakly compact iff it is $ao$-$wc$.
\end{corollary}
\begin{proof}

Let $E$ has an order unit and operator $T:E\to F$ be $ao$-$wc$, then it is a weakly compact operator.\\
Let $E^\prime$ be a $KB$-space, $F^\prime$ has an order unit and operator $T:E\to F$ is $ao$-$wc$. By Theorem \ref{eli}, $T^\prime$ is $ao$-$wc$. Because $F^\prime$ has an order unit, $T^\prime$ is weakly compact. By Theorem 5.5 of \cite{2}, $T$ is weakly compact.
\end{proof}
\begin{remark}
	Let  $E$ be a Banach lattice with an order unit. Then a subset $A$ of $E$ is norm bounded iff is order bounded iff it is almost order bounded.  Therefore an  operator $T:E\to F$ is weakly compact if and only   if is order weakly compact if and only if is $ao$-$wc$.
\end{remark}
\begin{remark}
	Under  conditions of Theorem \ref{eli}, if $T:E\to F$ is $ao$-$wc$, then by Corollary \ref{hg} and Theorem 5.44 of \cite{1}, there exist a reflexive Banach lattice $G$, lattice homomorphism $Q:E\to G$ and positive operator $S:G\to F$ that $T= S\circ Q$.
\end{remark}
Note that the identity operator $I:\ell^\infty \to \ell^\infty$ is not $ao$-$wc$, however its adjoint $I:(\ell^\infty)^\prime \to (\ell^\infty)^\prime$ is  $ao$-$wc$.

Let $T:E\to F$ be an operator between two Banach lattices. If $T^\prime :F^\prime \to E^\prime$ is $ao$-$wc$ and $F^\prime$ has an order unit, then $T^\prime$ is weakly compact and therefore $T$ is weakly compact. Hence $T$ is $ao$-$wc$.
If $T$ is $M$-weakly compact or $L$-weakly compact, then by Theorem 5.61 of \cite{1}, $T$ is weakly compact and therefore  $T$ is an $ao$-$wc$ operator. Thus we have the following result.
\begin{theorem}\label{fd}
	Let $T:E\to F$ be an operator between two Banach lattices. By one of the following conditions $T$ is an $ao$-$wc$ operator.
	\begin{enumerate}
		\item $T$ is $M$-weakly compact,
		\item $T$ is $L$-weakly compact, \\
		moreover if $F$ has order continuous norm and $T:E\to F$ is
		
	\end{enumerate}
\end{theorem}
If $T:E\to F$ is semicompact operator, or
dominated by a semicompact opereator, then  $T$ is an $ao$-$wc$.	
Let $A$ be an almost order bounded subset of $E$. Then $A$ is norm bounded. Therefore if $T$ is a semicompact operator, $T(A)$ is an almost order bounded in $F$. Since $F$ has order continuous norm, $T(A)$ is relatively weakly compact subset of $F$. Hence $T$ is an $ao$-$wc$ operator.
If $T$ is  dominated by a semicompact operator, then by Theorem 5.72 of \cite{1}, $T$ is semicompact operator and so is an $ao$-$wc$ operator.

\begin{remark}\label{rd}
	\begin{enumerate}
		\item An $ao$-$wc$ operator need not be an $M$-weakly or $L$-weakly compact  operator. For instance, the identity operator $I : L_1[0,1] \to L_1[0,1]$ is $ao$-$wc$, but is not $M$-weakly or $L$-weakly compact operator.
		\item Note that if $F$ has not order continuous norm, then each semicompact operator $T:E\to F$ is not necessarily $ao$-$wc$. For example, the identity operator $I:\ell^\infty \to \ell^\infty$ is semicompact and $\ell^\infty$ has not order continuous norm. Thus $I$ is not $ao$-$wc$.
	\end{enumerate}
\end{remark}
Let $E$, $F$ be two normed vector lattices. Recall from \cite{99}, a continuous operator $T:E\to F$ is said to be $\sigma$-$uon$-continuous, if for each norm bounded $uo$-null sequence $\{x_n\}\subseteq E$ implies $T(x_n)\xrightarrow{\|.\|}0$. 

\begin{theorem}\label{ree}
	Let $E$ and $F$ be two Banach lattices that $F$ is Dedekind complete. The positive operator $T:E\to F$ is $ao$-$wc$ iff for each disjoint almost order bounded sequence $\{x_n\}$ in $E$ implies $T(x_n) \xrightarrow{\|.\|}0$ in $F$.
\end{theorem}
\begin{proof}
	Let the operator $T:E\to F$ be  $ao$-$wc$. This means that for every $\epsilon$ there exists $u\in E^+$ such that $T([-u,u]+\epsilon B_E)$ is relatively weakly compact. Let $I_z$ be the ideal generated by $z\in[-u,u]+\epsilon B_E$ in $E$. The operator $T\vert _{I_z}:I_z \to F$ is weakly compact operator. Since $I_z$ is an $AM$-space with order unit, therefore $T\arrowvert_{I_z}:I_z\to F$ is $M$-weakly and hence by Remark 2.8 of \cite{99}, is $\sigma$-$uon$-continuous. It is clear that the extension of operator $T\vert _{I_z}$, $T:E\to F$ is $\sigma$-$uon$-continuous. If $\{x_n\}\subseteq E$ is almost order bounded and disjoint, hence it is norm bounded and $uo$-null. So we have $T(x_n)\xrightarrow{\|.\|}0$.	
	Conversely, let $ A\subseteq E$ be an almost order bounded set. Then for each $\epsilon$ there exists $u\in E^+$ such that $A\subseteq [-u,u]+\epsilon B_E$. Let $I_u$ be the ideal generated by $u$ in $E$ and $\{x_n\}\subseteq A$ be a disjoint sequence. It is clear that $\{x_n\}$ is norm bounded. By assumption, we have $T(x_n)\xrightarrow{\|.\|}0$ in $F$. Therefore $T:I_u \oplus E \to F$ is $M$-weakly compact, and so by Theorem \ref{fd}, $T:I_u \oplus E \to F$  is an $ao$-$wc$ operator. Thus $T:E\to F$ is $ao$-$wc$.
\end{proof}
\begin{corollary}
	\begin{enumerate}\label{elinn}
		\item Let $T:E\to F$, $S:F\to G$ be two $ao$-$wc$ operators where $F, G$ are Dedekind complete and $\{x_n\}\subseteq E$ is a disjoint almost order bounded sequence. By Theorem \ref{ree}, we have $T(x_n)\xrightarrow{\|.\|}0$. Since $S$ is a continuous operator, $ S(T(x_n))\xrightarrow{\|.\|}0$. Therefore $S\circ T$ is $ao$-$wc$ operator.
		\item	By Theorem 5.60 of \cite{1}, obviously that if $T:E\to F$ is an $ao$-$wc$ operator, then for each $\epsilon > 0$ there exists some $u\in E^+$ such that $\| T((|x|-u)^+)\|< \epsilon$ holds for all $x\in A$ that $A$ is an almost order bounded subset of $E$.
	\end{enumerate}
\end{corollary}
Recall that a Banach lattice $E$ is said to have the dual positive Schur property if every positive $w^*$-null sequence in $E^\prime$ is norm null.
\begin{theorem}\label{lk}
	The following statements are true.	
	\begin{enumerate}
		\item Let $E$  be a Banach lattice Dedekind complete. $E$ has order continuous norm iff each positive operator $T$ from $E$ into each Banach lattice $F$ is an $ao$-$wc$ operator.
		\item Let $E$  be a Banach lattice Dedekind complete. $E$ has order continuous norm iff each almost order bounded disjoint sequence $\{x_n\}\subseteq E$ is norm null.
		\item  If $E$ has the property $(b)$ and each operator $T^2:E\to E$ is $ao$-$wc$, then $E$ has order continuous norm.
		\item Let $T:E\to F$ be a continuous operator between two Banach lattices $E,F$ that $F$ is Dedekind complete. If $|T| $ exists and it is $ao$-$wc$, then $T$ is also $ao$-$wc$.
		\item If $E$ has the dual positive Schur propertry, $F$ has order continuous norm and Dedekind complete, then adjoint of each positive operator $T:E\to F$ is an $ao$-$wc$ operator.
	\end{enumerate}
\end{theorem}
\begin{proof}
	\begin{enumerate}
		\item Let $E$ has order continuous norm and $\{x_n\}$ be an almost order bounded disjoint sequence in $E$. Therefore $x_n\xrightarrow{uo}0$ in $E$. By Proposition 3.7 of \cite{9c}, $x_n \xrightarrow{\|.\|}0$. By continuity of $T$, it follows that  $Tx_n\xrightarrow{\|.\|}0$ in $F$.
		
		Conversely, let $E$ has not order continuous norm. By Theorem 2.7 of \cite{11}, there exists an operator $T$ from $E$ into $\ell^\infty$ such that  $T$ is not order weakly compact and therefore  is not $ao$-$wc$. 
		\item Let $E$ has order continuous norm, therefore the identity operator $I:E\to E$ is $ao$-$wc$.  Then $x_n = Ix_n\xrightarrow{\|.\|}0$ where $\{x_n\}\subseteq E$ is almost order bounded disjoint sequence.
		
		Conversely, let $\{x_n\}$ be an order bounded disjoint sequence in $E$. Therefore $\{x_n\}$ is almost order bounded disjoint in $E$. Hence by assumption $x_n = Ix_n\xrightarrow{\|.\|}0$. By Theorem 4.14 of \cite{1}, $E$ has order continuous norm.
		\item By contradiction, assume that  $E$ has not order continuous norm, it follows from the proof of Theorem 2 of \cite{32},
		that $E$ contains a closed order copy of $c_0$ and there exists a positive projection $P : E \to c_0$.
		Let $i : c_0 \to E$ be the canonical injection.
		Obviously that $T = i \circ P : E \to E$ is not $b$-weakly compact. Since $E$ has property $(b)$, therefore $T$ is not order weakly compact, and so  $T^2$ is not $ao$-$wc$.
		\item Let $0\leq T \leq S$ and $S$ be an $ao$-$wc$ operator. If $\{x_n\} $ is an almost order bounded and disjoint sequence in $E$, then by Theorem \ref{ree}, $S(x_n)\xrightarrow{\|.\|}0$. Therefore $T(x_n)\xrightarrow{\|.\|}0$. We have $-|T|\leq T \leq |T|$ and so $0\leq T+|T|\leq 2|T|$. It follows that  $T$ is an $ao$-$wc$ operator whenever  $|T|$ is $ao$-$wc$. 
		\item 	Let $\{f_{n}\} $ be an almost order bounded disjoint sequence in $F^\prime$.  Then $ f_{n} \xrightarrow{uo} 0$ in $F^\prime$. Without loss of generality, assume $0\leq f_{n}$. Note that $ 0\leq T^{\prime} f_{n}$.  Now since $ F$ has order continuous norm, by Theorem 2.1 from \cite{7c}, $f_{n} \xrightarrow{w^{*}} 0$ in $ F^\prime$. Since $T^{\prime} $ is $w^{*}$-to-$w^{*}$ continuous, hence $T^{\prime} f_{n} \xrightarrow{w^{*}} 0$ in $E^{\prime}$. Since $E$ has the dual positive Schur property, hence $ T^{\prime} f_{n} \xrightarrow{\|.\|} 0$ in $E^{\prime}$.
	\end{enumerate}
\end{proof} 
\begin{proposition}
	If $E$ has an order unit, then $T:E\to F$ is $\sigma$-$uon$-continuous if and only if it is an $ao$-$wc$ operator. 
\end{proposition}
\begin{proof}
	Let $T:E\to F$ be an $ao$-$wc$ operator, then $T$ is order weakly compact. Let $\{x_n\}\subseteq E$ be a norm bounded disjoint sequence. Since $E$ has an order unit, then $\{x_n\}$ is order bounded disjoint sequence. By assumption and  Theorem 5.57 of \cite{1}, $T(x_n)\xrightarrow{\|.\|}0$. So $T$ is $M$-weakly compact and therefore by remark 2.8 of \cite{99}, $T$ is $\sigma$-$uon$-continuous.
\end{proof}
In the following, we  establish some relationships between the class of $ao$-$wc$ operators and the class of semicompact operators. By Remark \ref{rd}, we know that the class of $ao$-$wc$ operators different with the class of semicompact operators, but as following we see some relations.
\begin{theorem}\label{hgg}
	Let $T:E\to F$ be an $ao$-$wc$ operator between two Banach lattices. Then $T$ is semicompact operator.
\end{theorem}
\begin{proof}
	Let $T:E\to F$ be $ao$-$wc$. Let $A$ be an almost order bounded subset of $E$.  Without loss of generality we assume that for each $\epsilon$ there exists $u\in E^+$ such that $A = [-u,u] +\epsilon B_E$. Let $p(x) = \|x\|$. Then $\lim p(Tx_n) = 0$ holds for each disjoint sequence $\{x_n\}$ in $A$. By Theorem 4.36 of \cite{1}, there exists some $v\in E^+$ satisfying $\|T(\vert  x\vert  -v)^+\|\leq \epsilon$ for all $x\in A$. Put $w = Tv\in F^+$, and note that
	$$(\vert  Tx\vert  -w)^+ =$$
	$$(\vert  Tx\vert  -Tv)^+ \leq $$
	$$ (T\vert  x\vert  -Tv)^+=$$
	$$ (T(\vert  x\vert  -v))^+\leq$$
	$$ T((\vert  x\vert -v)^+).$$
	 Therefore $T$ is a semicompact operator.
\end{proof}
By Theorems \ref{fd} and \ref{hgg}, we have the following result.
\begin{corollary}
	\begin{enumerate}
		\item Each operator $T:E\to F$ that it is $ao$-$wc$ is an almost order bounded operator.
\item	Let $F$ be a Banach lattice with order continuous norm. Then  $T:E\to F$ is $ao$-$wc$ if and only if it is a semicompact operator.
\end{enumerate}
\end{corollary}
If $T$ is $ao$-$wc$, in general $\vert  T\vert  $ is not exist, see the following example.
\begin{example} The operator $T:L_1[0,1]\to c_0$ defined by 
	$$ T(f) = (\int_0 ^1 f(x)\sin x dx, \int_0 ^1 f(x)\sin2x dx,\cdots),$$
	is an $ao$-$wc$ operator. Note that by Exercise 10 of page 289 of \cite{1}, its modulus does not exist.
\end{example}
In the following theorem, under some conditions, we show that $\vert  T\vert  $ exist and is $ao$-$wc$ whenever $T$ is $ao$-$wc$.

Recall that a Banach lattice $E$ is said to have the
property $(P)$ if there exists a positive contractive projection $P : E^{\prime\prime} \to E$ where $E$
is identified with a sublattice of its topological bidual $E^{\prime\prime}$.
\begin{theorem}\label{el}
	Let $T:E\to F$ be an $ao$-$wc$ operator. By one of the following conditions, the modulus of $T$ exists and it is an $ao$-$wc$ operator.
\end{theorem}
\begin{enumerate}
	\item $E$ is an $AL$-space and $F$ has the propery $(P)$.
	\item  $E$ and $F$ have order unit.
	\item $F$ is Archimedean Dedekind complete and $ T $ is an order bounded preserves disjointness.
\end{enumerate}
\begin{proof}
	\begin{enumerate}
		\item By Theorem 1.7 of \cite{shef}, we have $L^r(E,F) = L(E,F)$. Therefore $\vert  T\vert $ exists. Since $E$ has order continuous norm,  by Theorem \ref{lk}, $\vert  T\vert :E\to F$ is an $ao$-$wc$ operator.
		\item Since $E$ has an order unit,  $T$ is a weakly compact operator. Since $F$ has an order unit, therefore by Theorem 2.3 of \cite{15}, the modulus of $T$ exists and it is a weakly compact operator. It is obvious that $\vert  T\vert $ is an $ao$-$wc$ operator.
		\item  By Theorem 2.40 of \cite{1}, 
		$ \vert  T\vert  $ exists and for all $x$, we have $\vert  T\vert  (\vert  x\vert ) =  \vert  T(\vert  x\vert )\vert  = \vert  T(x)\vert $.	
		If $\{x_n\} \subseteq E$ is an almost order bounded disjoint sequence, then by assumption $T(x_n)\xrightarrow{\|.\|}0$. 
		For each $n$, we have
		\(
		\vert  T\vert  (\vert  x_n\vert )=\vert  T(\vert  x_n\vert )\vert =\vert  T(x_n)\vert  \xrightarrow{\|.\|} 0
		\)
		in $ F$. The inequality $ \vert  (\vert  T \vert  (x_{n})) \vert \leq \vert  T \vert  \vert  x_n \vert  $,
		implies that   $$ \vert  T\vert  (x_n) \xrightarrow{\|.\|} 0.$$ Hence $\vert  T\vert $ is an $ao$-$wc$ operator.
	\end{enumerate}
\end{proof}
\begin{theorem}
	Let $E$ and $ F$ have order unit with  $F$  Dedekind complete. Then  $K_{ao-wc}(E,F)\cap L_b(E,F)$ is a band in $L_b(E,F)$.
\end{theorem}
\begin{proof}
	It is obvious that if $T, S \in K_{ao-wc}(E,F)\cap L_b(E,F)$ and $\alpha \in \mathbb{R}$, then $T+S,   \alpha T	\in K_{ao-wc}(E,F)\cap L_b(E,F)$. 
	
	Let $\vert  S\vert   \leq \vert  T\vert  $ where $T\in K_{ao-wc}(E,F)\cap L_b(E,F)$, $S\in L_b(E,F)$ and $\{x_n\}\subseteq E$ be almost order bounded disjoint sequence. Without loss of generality, assume that $x_n\geq 0$ for all $n$. By Theorem \ref{el}, $\vert  T\vert  (x_n)\xrightarrow{\|.\|}0$. The inequalities $\vert  S(x_n)\vert \leq \vert  S\vert (x_n)\leq\vert  T\vert  (x_n)$ implies that $S(x_n)\xrightarrow{\|.\|}0$. Therefore $S\in K_{ao-wc}(E,F)\cap L_b(E,F)$, and so $K_{ao-wc}(E,F)\cap L_b(E,F)$ is an ideal of $L_b(E,F)$.
	
	Now let $0\leq T_{\alpha}\uparrow T$ in $L_b(E,F)$ with $\{T_{\alpha}\}\subseteq K_{ao-wc}(E,F)\cap L_b(E,F)$. Since $T$ is positive, therefore $T$ is order bounded and since $E$ has an order unit, then by Example \ref{elin}, $T$ is $ao$-$wc$. Hence $T\in K_{ao-wc}(E,F)\cap L_b(E,F)$.
\end{proof}

\end{document}